\documentclass[12pt]{article}
\usepackage{amssymb}
\usepackage{amsmath,amsthm}
\usepackage[latin1]{inputenc}
\usepackage{hyperref}
\usepackage{enumerate}
\usepackage{tikz}
\usepackage{tkz-graph}

\hypersetup{colorlinks=true}

\hypersetup{colorlinks=true, linkcolor=blue, citecolor=blue,urlcolor=blue}

 \newtheorem{remark}{Remark}

 \newtheorem{lemma}[remark]{Lemma}
 \newtheorem{theorem}[remark]{Theorem}
 \newtheorem{proposition}[remark]{Proposition}
 \newtheorem{corollary}[remark]{Corollary}

\title{On the strong metric dimension of Cartesian sum graphs}

\author{Dorota Kuziak$^{(1)}$, Ismael G. Yero$^{(2)}$ and Juan A. Rodr\'{\i}guez-Vel\'{a}zquez$^{(1)}$\\
$^{(1)}${\small Departament d'Enginyeria Inform\`atica i Matem\`atiques,}\\
{\small Universitat Rovira i Virgili,} {\small Av. Pa\"{\i}sos Catalans 26, 43007 Tarragona, Spain.}\\
{\small dorota.kuziak\@@urv.cat, juanalberto.rodriguez\@@urv.cat}\\
$^{(2)}${\small Departamento de Matem\'aticas, Escuela Polit\'ecnica Superior de Algeciras}\\
{\small Universidad de C\'adiz,} {\small Av. Ram\'on Puyol s/n, 11202 Algeciras, Spain.}\\
{\small ismael.gonzalez\@@uca.es}\\
}

\begin{document}
\maketitle

\begin{abstract}
A vertex $w$ of a connected graph $G$ strongly resolves two vertices $u,v\in V(G)$, if there exists some shortest $u-w$ path containing $v$ or some shortest $v-w$ path containing $u$. A set $S$ of vertices is a strong metric generator for $G$ if every pair of vertices of $G$ is strongly resolved by some vertex of $S$. The smallest cardinality of a strong metric generator for $G$ is called the strong metric dimension of $G$. In this paper we obtain several tight bounds or closed formulae for the strong metric dimension of the Cartesian sum of graphs in terms of the strong metric dimension, clique number or twins-free clique number of its factor graphs.
\end{abstract}



\section{Introduction}

Nowadays several applications of locating sets for graphs can be found in literature. For instance, applications to long range aids to navigation \cite{Slater1975}; to chemistry for representing chemical compounds \cite{Johnson1993, Johnson1998}; to problems of pattern recognition and image processing  \cite{Melter1984}; or to navigation of robots in networks \cite{Khuller1996}. Nevertheless, the first problem which motivated the definition of locating sets was the problem of uniquely recognizing the position of an intruder in a network, described by Slater in \cite{Slater1975}. Also, an equivalent concept was  introduced independently by Harary and Melter in \cite{Harary1976}, where the locating sets were called resolving sets. Moreover, in accordance with the well-known terminology of metric spaces, in \cite{Sebo2004} locating sets (resolving sets) were renamed as metric generators. In this work we follow the terminology of metric generator. After those primary articles, several variants of metric generators have been appearing in the graph theory researches. In this sense, according to the amount of literature concerning this topic and all its close variants, we restrict our references to those ones that we really refer to in a non-superficial way.

One of the variants of metric generator was presented by Seb\H{o} and Tannier in \cite{Sebo2004}, and studied further in several articles. Given a connected graph $G=(V(G),E(G))$ and two vertices $x,y\in V(G)$, the distance between $x$ and $y$ is the length of a shortest $x-y$ path in $G$ and is denoted by $d_G(x,y)$. A vertex $w\in V(G)$ \emph{strongly resolves} two vertices $u,v\in V(G)$ if $d_G(w,u)=d_G(w,v)+d_G(v,u)$ or $d_G(w,v)=d_G(w,u)+d_G(u,v)$, \emph{i.e.}, there exists some shortest $w-u$ path containing $v$ or some shortest $w-v$ path containing $u$. A set $S$ of vertices of $G$ is a \emph{strong metric generator} for $G$ if every two vertices of $G$ are strongly resolved by some vertex of $S$. The smallest cardinality of a strong metric generator of $G$ is called the \emph{strong metric dimension} and is denoted by $\dim_s(G)$. A \emph{strong metric basis} of $G$ is a strong metric generator for $G$ of cardinality $\dim_s(G)$.

Graph products have been frequently investigated in the last few years and the theory of recognizing the structure of classes of these
graphs is very common. For more information in this topic we suggest the book \cite{Hammack2011}. Other standard approach to graph products is to deduce properties of the product with respect to its factors. The case of strong metric dimension has not escaped to these kind of studies. For example, the strong metric dimension of product graphs has been studied for   Cartesian product graphs and direct product graphs \cite{Rodriguez-Velazquez2013a}, for strong product graphs \cite{Kuziak2013c,Kuziak-Erratum}, for corona product graphs and join graphs \cite{Kuziak2013}, for rooted product graphs \cite{Kuziak2013b} and for lexicographic product graphs \cite{Kuziak2014}. In this paper we study the strong metric dimension of Cartesian sum graphs.

Now we give some basic concepts and notations. Let $G=(V,E)$ be a simple graph. For two adjacent vertices $u$ and $v$ of $G$ we use the notation  $u\sim v$ and, in this case, we say that $uv$ is an edge of $G$, \emph{i.e.}, $uv\in E$. The complement $G^c$ of $G$ has the same vertex set than $G$ and $uv\in E(G^c)$ if and only if $uv\notin E$. The diameter of $G$ is defined as $$D(G)=\max_{u,v\in V}\{d_G(u,v)\}.$$
If $G$ is not connected, then we assume that the distance between any two vertices belonging to different components of $G$ is infinity and, thus, its diameter is $D(G)=\infty$. For a vertex $v\in V,$ the set $N_G(v)=\{u\in V:\; u\sim v\}$ is the open neighborhood of $v$ and the set $N_G[v] = N_G(v)\cup \{v\}$ is the closed neighborhood of $v$. We recall that the \emph{clique number} of a graph $G$, denoted by $\omega(G)$, is the number of vertices in a maximum clique in $G$. We refer to an $\omega(G)$-set in a graph $G$ as a clique of cardinality $\omega(G)$. Two vertices $x$, $y$ are called \emph{true twins} if $N_G[x] = N_G[y]$. We say that $X\subset V(G)$ is a \emph{twins-free clique} in $G$ if $X$ is a clique and for every $u,v\in X$ it follows $N_G[u]\ne N_G[v]$, \emph{i.e.}, $X$ is a clique and it contains no true twins. We say that the \emph{twins-free clique number} of $G$, denoted by $\varpi(G)$, is the maximum cardinality among all twins-free cliques in $G$. Thus, $\omega(G)\ge \varpi(G)$. We refer to an $\varpi(G)$-set in a graph $G$ as a twins-free clique of cardinality $\varpi(G)$.

A set $S$ of vertices of $G$ is a \emph{vertex cover} of $G$ if every edge of $G$ is incident with at least one vertex of $S$. The \emph{vertex cover number} of $G$, denoted by $\beta(G)$, is the smallest cardinality of a vertex cover of $G$. We refer to an $\beta(G)$-set in a graph $G$ as a vertex cover set of cardinality $\beta(G)$.

Recall that the largest cardinality of a set of vertices of $G$, no two of which are adjacent, is called the \emph{independence number} of $G$ and is denoted by $\alpha(G)$. We refer to a $\alpha(G)$-set in a graph $G$ as an independent set of cardinality $\alpha(G)$. The following well-known result, due to Gallai, states the relationship between the independence number and the vertex cover number of a graph.

\begin{theorem}{\rm (Gallai's theorem)}\label{th gallai}
For any graph  $G$ of order $n$,
$$\beta(G)+\alpha(G) = n.$$
\end{theorem}

A vertex $u$ of $G$ is \emph{maximally distant} from $v$ if for every $w\in N_G(u)$, $d_G(v,w)\le d_G(u,v)$. If $u$ is maximally distant from $v$ and $v$ is maximally distant from $u$, then we say that $u$ and $v$ are \emph{mutually maximally distant}. The {\em boundary} of $G=(V,E)$ is defined as $\partial(G) = \{u\in V:\; \mbox{exists } v\in V\, \mbox{such that } u,v\mbox{ are mutually maximally distant}\}$.
We use the notion of strong resolving graph introduced by Oellermann and Peters-Fransen in \cite{Oellermann2007}. The \emph{strong resolving graph}\footnote{In fact, according to \cite{Oellermann2007} the strong resolving graph $G'_{SR}$ of a graph $G$ has vertex set $V(G'_{SR})=V(G)$ and two vertices $u,v$ are adjacent in $G'_{SR}$ if and only if $u$ and $v$ are mutually maximally distant in $G$. So, the strong resolving graph defined here is a subgraph of the strong resolving graph defined in \cite{Oellermann2007} and can be obtained from the latter graph by deleting its isolated vertices.} of $G$ is a graph $G_{SR}$  with vertex set $V(G_{SR}) = \partial(G)$ where two vertices $u,v$ are adjacent in $G_{SR}$ if and only if $u$ and $v$ are mutually maximally distant in $G$.

If it is the case, for a non-connected graph $G$ we use the assumption that any two vertices belonging to different components of $G$ are mutually maximally distant between them.

It was shown in  \cite{Oellermann2007}  that the problem of finding the strong metric dimension of a graph $G$ can be transformed into the problem of computing the vertex cover number of $G_{SR}$.

\begin{theorem}{\em \cite{Oellermann2007}}\label{th oellermann}
For any connected graph $G$,
$$\dim_s(G) = \beta(G_{SR}).$$
\end{theorem}

We use the notation $K_n$,  $C_n$, $N_n$ and $P_n$ for complete graphs,  cycle graphs, empty graphs and path graphs, respectively. Moreover, any graph with at least two vertices is a non-trivial graph, while an empty graph is an edgeless non-trivial graph. In this work, the remaining definitions are given the first time that the concept appears in the text.

\section{Cartesian sum graphs}

The \textit{Cartesian sum} of two graphs $G=(V_1,E_1)$ and $H=(V_2,E_2)$, denoted by $G\oplus H$, is the graph with vertex set $V=V_1\times V_2$, where $(a,b)(c,d)\in E(G\oplus H)$ if and only if $ac\in E_1$ or $bd\in E_2$.

This notion of graph product was introduced by Ore \cite{Ore1962} in 1962, nevertheless it has passed almost unnoticed and just few results (for instance \cite{Cizek1994,Scheinerman1997}) have been presented about this. According to that we consider it deserves to begin the study of some of its properties. The Cartesian sum is also known as the \emph{disjunctive product} \cite{Scheinerman1997} and it is a commutative operation \cite{Hammack2011}. This well known fact is very useful in this section. Moreover, in several results, symmetric cases are omitted without specific mentioning of that fact.

The \textit{lexicographic product} of two graphs $G=(V_1,E_1)$ and $H=(V_2,E_2)$ is the graph $G\circ H$ with vertex set $V=V_1\times V_2$ and two vertices $(a,b),(c,d)\in V$  are adjacent in $G\circ H$ if and only if either $ac\in E_1$, or $a=c$ and $bd\in E_2$.

Note that the lexicographic product of two graphs is not a commutative operation. Moreover, $G\circ H$ is a connected graph if and only if $G$ is connected. We would point out the following fact.

\begin{remark}\label{rem cart-lex}
For any graph  $G$ and any non-negative integer $n$,  $$G\oplus N_n= G\circ N_n.$$
\end{remark}

The strong metric dimension of $G\circ N_n$ was studied in \cite{Kuziak2014}. In order to present some results on  $\dim_s(G\circ N_n)$, we need to introduce some additional notation and terminology.
We define the {\em TF-boundary} of a  non-complete graph $G=(V,E)$ as a set $\partial_{TF}(G) \subseteq \partial(G) $ where $x\in \partial_{TF}(G)$ whenever there exists $y\in \partial (G)$ such that $x$ and $y$ are mutually maximally distant in $G$ and $N_G[x]\ne N_G[y]$ (which means that $x,y$ are not true twins).
The \emph{strong resolving TF-graph} of $G$ is a graph $G_{SRS}$  with vertex set $V(G_{SRS}) = \partial_{TF}(G)$ where two vertices $u,v$ are adjacent in $G_{SRS}$ if and only if $u$ and $v$ are mutually maximally distant in $G$ and $N_G[x]\ne N_G[y]$. Notice that the strong resolving TF-graph is a subgraph of the strong resolving graph.

\begin{proposition}{\rm \cite{Kuziak2014}}\label{H=empty}
Let $G$ be a connected non-complete  graph of order $n\ge 2$ and let  $n'\ge 2$ be an integer.
Then
$$\dim_s(G\circ N_{n'})=n(n'-1)+\beta(G_{SRS}).$$
In particular, if $G$ has no true twin vertices, then $$\dim_s(G\circ N_{n'})=n(n'-1)+\dim_s(G).$$
Moreover,
$$\dim_s(K_n\circ N_{n'})=n(n'-1).$$
\end{proposition}

The following remark is a direct consequence of the definition of Cartesian sum graph.
\begin{remark}\label{rem diam1}
A graph $G\oplus H$ is complete if and only if both, $G$ and $H$, are complete graphs.
\end{remark}

Because of the above we continue with the Cartesian sum of two  graphs $G$ and $H$, such that $G$ or $H$ is not complete.

\begin{proposition}\label{lem Cart sum diam}
Let $G$ and $H$ be two non-trivial graphs such that at least one of them is non-complete and let $n\ge 2$ be an integer. Then the following assertions hold.
\begin{enumerate}[{\rm (i)}]
\item  $D(G\oplus N_n)=\max\{2,D(G)\}.$
\item If $G$ and $H$ have isolated vertices, then  $D(G\oplus H)=\infty$.
\item If neither $G$ nor  $H$ has isolated vertices, then $D(G\oplus H)=2$.
\item If  $D(H)\le 2$, then $D(G\oplus H)=2$.
\item If $D(H)>2$, $H$ has no isolated vertices and $G$ is a non-empty graph having at least one isolated vertex, then $D(G\oplus H)=3$.
\end{enumerate}
\end{proposition}

\begin{proof}Note that since $G$ and $H$ are two graphs such that at least one of them is non-complete, by Remark \ref{rem diam1}  we have that $D(G\oplus H)\ge 2$.
\begin{enumerate}[{\rm (i)}]

\item  If $G$ is connected, then $d_{G\oplus N_{n}}((a,b),(c,d)) = d_{G}(a,c)$ and  $d_{G\oplus N_{n}}((a,b),(a,d)) = 2$. Thus,
$D(G\oplus N_n)=\max\{2,D(G)\}.$

On the other hand, if $G_1$ and $G_2$ are two connected components of $G$, then for any $u\in V(G_1)$, $x\in V(G_2)$ and $v,y\in V(N_{n})$, we have that $(u,v)\not\sim (x,y)$, so $G\oplus N_{n}$ is not connected and, as a result,  $D(G\oplus N_n)=\infty$.
\item
If $u\in V(G)$ and $v\in V(H)$ are  isolated vertices, then $(u,v)\in V(G\oplus H)$ is an  isolated vertex, so  (ii) follows.

\item Assume that  neither $G$ nor  $H$ has isolated vertices. We consider the following cases for  two different vertices  $(u,v),(x,y)\in V(G\oplus H)$.
\\
\\
\noindent Case 1: $v=y$. Since $H$ has no isolated vertices, then there exists a vertex $w\in N_H(v)$. So,  $(u,v)\sim (u,w)\sim (x,y)$ and, as a consequence, $d_{G\oplus H}((u,v),(x,y))\le 2$.
\\
\\
\noindent Case 2: $u=x$. This case is symmetric to Case 1.
\\
\\
\noindent Case 3: $v\ne y$ and $u\ne x$. Since $G$ and $H$ have no isolated vertices,  there exist vertices $z\in N_G(x)$ and $w\in N_H(v)$. Hence,  $(u,v)\sim (z,w)\sim (x,y)$ and, as a result, $d_{G\oplus H}((u,v),(x,y))\le 2$.

According to the cases above the proof of (iii) is complete.

\item Let  $D(H)\le 2$. If $v$ and $y$ are two adjacent vertices of  $H$, then for any $u,x\in V(G)$ we have $d_{G\oplus H}((u,v),(x,y))=1$, while  if $v\not\sim y$ ($v$ and $y$ are not necessarily different), then for any $w\in N_H(v)\cap N_H(y)$ we have $(u,v)\sim (x,w)\sim (x,y)$. Thus, $d_{G\oplus H}((u,v),(x,y))\le 2$ and so (iv) follows.

\item Assume that $G$ has an isolated vertex, $H$ has no isolated vertices and $D(H)>2$. If $u$ and $x$ are not isolated vertices in $G$, then we proceed like in the proof of (iii) to show that $d_{G\oplus H}((u,v),(x,y))\le 2$. If $u$ or $x$ is an isolated vertex of $G$ and $d_H(v,y)\le 2$, then we proceed like in the proof of (iv). So, we consider that $u$ or $x$ is an isolated vertex and $d_H(v,y)\ge 3$.\\
\\
\noindent Case 1': $u$ is an isolated vertex and $x$ is not an isolated vertex. In this case there exists $t\in N_G(x)$ and,  since $H$ has no isolated vertices, there exists $w\in N_H(v)$. Hence, $(u,v)\sim (t,w)\sim (x,y)$ and, as a consequence,  $d_{G\oplus H}((u,v),(x,y))\le 2$.
\\
\\
\noindent Case 2':  $u$ and $x$ are isolated vertices ($u$ and $x$ are not necessarily different). Since $H$ has no isolated vertices and $d_H(v,y)\ge 3$, for every two vertices $w\in N_H(v)$ and $z\in N_H(y)$ it follows that $w\ne z$. Moreover, since $G$ is not  empty, there exist two different vertices $s,t\in V(G)$ such that $s\sim t$. Hence,  $(u,v)\sim (t,w)\sim (s,z)\sim (x,y)$. Thus, $d_{G\oplus H}((u,v),(x,y))\le 3$. On the other hand, since $N_{G\oplus H}(u,v)=V(G)\times N_H(v)$, $N_{G\oplus H}(x,y)=V(G)\times N_H(y)$ and $N_H(v)\cap N_H(y)=\emptyset$, we obtain that $N_{G\oplus H}(u,v)\cap N_{G\oplus H}(x,y)=\emptyset$.
Therefore, $d_{G\oplus H}((u,v),(x,y))= 3$ and the proof of (v) is complete.
\end{enumerate}
\end{proof}

\begin{corollary}\label{rem diam2}
The graph $G\oplus H$ is not connected if and only if both $G$ and $H$ have isolated vertices or $G$ is an empty graph and $H$ is not connected.
\end{corollary}

In order to present  the next result we need to introduce some more terminology. Given a graph  $G$, we  define $G^*$ as the graph with vertex set $V(G^*)=V(G)$ such that  two vertices $u,v$ are adjacent in $G^*$ if and only if either $d_G(u,v)\ge 2$ or  $u,v$ are true twins. If a graph $G$ has at least one isolated vertex, then we denote by 
$G_-$ the graph obtained from $G$ by removing all its isolated vertices. In this sense, $G^*_-$ is obtained from $G^*$ by removing all its isolated vertices.  Notice  that if $G$ has no true twins, then $G^*\cong G^c$.

\begin{proposition}\label{SRGraphDianCartSumle2}
Let $G$ and $H$ be two non-trivial graphs such that at least one of them is non-complete. If   $D(G)\le 2$
or neither $G$ nor  $H$ has isolated vertices, then
$$(G\oplus H)_{SR}\cong (G\oplus H)^*_-.$$
\end{proposition}

\begin{proof}
We assume that $D(G)\le 2$
or neither $G$ nor  $H$ has isolated vertices. Then, by Proposition \ref{lem Cart sum diam} we have $D(G\oplus H)=2$ and, as a consequence,  two vertices are mutually  maximally distant in $G\oplus H$ if and only if they are true twins or they are not adjacent. Hence, $(G\oplus H)_{SR}\cong (G\oplus H)^*_-$.
\end{proof}

Our next result is derived from Theorem \ref{th oellermann} and Proposition \ref{SRGraphDianCartSumle2}.

\begin{proposition} \label{StrongDimDiamle2}
Let $G$ and $H$ be two non-trivial graphs such that at least one of them is non-complete. If   $D(G)\le 2$
or neither $G$ nor  $H$ has isolated vertices, then
$$\dim_s(G\oplus H) = \beta((G\oplus H)^*_-).$$
\end{proposition}

The \emph{strong product} of two graphs $G=(V_1,E_1)$ and $H=(V_2,E_2)$ is the graph $G\boxtimes H=(V,E)$, such that $V=V_1\times V_2$ and two vertices $(a,b),(c,d)\in V$ are adjacent in $G\boxtimes H$ if and only if
\begin{itemize}
\item[] $a=c$ and $bd\in E_2$, or
\item[] $b=d$ and $ac\in E_1$, or
 \item[] $ac\in E_1$ and $bd\in E_2$.
\end{itemize}
We would point out  that the Cartesian product $G\square  H$ is a subgraph of $G\boxtimes H$ and for complete graphs it holds $K_r\boxtimes K_s=K_{rs}$.

\begin{lemma}\label{complementCartSum}
For any graphs $G$ and $H$,
$$(G\oplus H)^c=G^c\boxtimes H^c.$$
\end{lemma}

\begin{proof}
Two vertices $(u,v)$ and $(u',v')$ are adjacent in  $(G\oplus H)^c$ if and only if ($u$ and $u'$ are not adjacent in $G$) and ($v$ and $v'$ are not adjacent in $H$). \textit{i.e.}, $(u,v)$ and $(u',v')$ are adjacent in  $(G\oplus H)^c$ if and only if
\begin{itemize}
\item[] $u=u'$ and $v\sim v'$ in $H^c$, or
\item[] $u\sim u'$ in $G^c$ and $v=v'$, or
 \item[] $u\sim u'$ in $G^c$ and $v\sim v'$ in $H^c$.
\end{itemize}
Therefore,  $(G\oplus H)^c=G^c\boxtimes H^c.$
\end{proof}

\begin{theorem}\label{dims-cover-complement}
Let  $G$ and $H$ be two graphs of order $n$ and $n'$, respectively. If {\rm (}neither $G$ nor $H$ has true twin vertices{\rm )}  and {\rm(}$D(G)\le 2$
or neither $G$ nor  $H$ has isolated vertices{\rm)}, then
$$\dim_s(G\oplus H) = \beta( G^c\boxtimes H^c).$$
\end{theorem}

\begin{proof}
If   $D(G)\le 2$ or neither $G$ nor  $H$ has isolated vertices, then by  Proposition \ref{StrongDimDiamle2} we have
$\dim_s(G\oplus H) = \beta((G\oplus H)^*_-)$.

Now, for any $(u,v)\in V(G\oplus H)$ we have
$$N_{G\oplus H}[(u,v)]=(N_G(u)\times V(H))\cup (V(G)\times N_H(v))\cup \{(u,v)\},$$
Hence, if neither $G$ nor  $H$ has  true twins,  then $G\oplus H$ has no  true twins and, as a result,
$(G\oplus H)^*_-=(G\oplus H)^c_-$.  Therefore,  we conclude the proof by  Lemma \ref{complementCartSum}, \textit{i.e.}, $\dim_s(G\oplus H) = \beta((G\oplus H)^*_-)=\beta((G\boxtimes H)^c_-)=\beta(( G^c\boxtimes H^c)_-)=\beta( G^c\boxtimes H^c). $
\end{proof}

Notice that transforming the problem of computing the strong metric dimension of the Cartesian sum of graphs into computing the vertex cover number of the strong product of the complement of the factor graphs (for the specific conditions of Theorem \ref{dims-cover-complement}), which is equivalent to obtain the independence number of such an strong product, is related to the well known Shannon capacity of a graph (see \cite{Shannon1956}). According to this, it is already known that obtaining the independence number of the strong product of graphs is a really challenging problem. In this sense, it seems to be very hard to give some examples of useful applications of the result above.

To do this, we need to introduce the following family of graphs presented previously in \cite{Kuziak2013c}. A \emph{$\mathcal{C}$-graph} is a graph $G$ whose vertex set can be partitioned into $\alpha(G)$ cliques. Notice that there are several graphs which are $\mathcal{C}$-graphs. For instance, we emphasize the following cases: complete graphs, cycles of even order or the complement of a cycle of even order. The following result on the independence number of the strong product of a $\mathcal{C}$-graph and any arbitrary graph was also presented in \cite{Kuziak2013c}.

\begin{lemma}{\em \cite{Kuziak2013c}}\label{lem Cgraph}
For any $\mathcal{C}$-graph $G$ and any graph $H$,
$$\alpha(G\boxtimes H) = \alpha(G)\alpha(H).$$
\end{lemma}

Theorems \ref{th gallai} and \ref{dims-cover-complement}, and Lemma \ref{lem Cgraph} lead to the next result, which is  an example of the usefulness of Theorem \ref{dims-cover-complement}.

\begin{remark}
Let  $G$ and $H$ be two graphs of order $n$ and $n'$, respectively. If {\rm (}neither $G$ nor $H$ has true twin vertices{\rm )}, {\rm(}$D(G)\le 2$ or neither $G$ nor  $H$ has isolated vertices{\rm)} and the complement of $G$ is a $\mathcal{C}$-graph, then
$$dim_s(G\oplus H)=n n'-\omega(G)\omega(H).$$
\end{remark}

\begin{proof}
By Theorem \ref{dims-cover-complement} we have that $\dim_s(G\oplus H) = \beta( G^c\boxtimes H^c)$. Now, by Theorem \ref{th gallai} and Lemma \ref{lem Cgraph} we obtain $\dim_s(G\oplus H) =nn' - \alpha(G^c\boxtimes H^c)=nn'-\alpha(G^c)\alpha(H^c)$. Therefore, the result follows, as the independence number of a graph equals the clique number of its complement.
\end{proof}

One example for the remark above could be, for instance, the case in which $G$ is a cycle of even order, since its complement is a $C$-graph (if $G$ is the complement of a cycle, then it happens the same). These examples are also presented after in Remark \ref{examples}.

We continue now with some relationships between the strong metric dimension and the twins-free clique number of the Cartesian sum of graphs. In this sense, the following theorem is also an important tool.

\begin{theorem}{\em \cite{Kuziak2013}}\label{th D2}
Let $G$ be a connected graph of order $n\ge 2$. Then
$$\dim_s(G)\le n-\varpi(G).$$
Moreover, if $D(G)=2$, then $$\dim_s(G)= n-\varpi(G).$$
\end{theorem}

The next relationship between the twins-free clique number of a Cartesian sum graphs and the twins-free clique number of its factors is also useful in our work.

\begin{lemma}\label{lem varpi}
Let $G$ and $H$ be two graphs. Then,
$$\varpi(G\oplus H)\ge \varpi(G) \varpi(H).$$
\end{lemma}

\begin{proof}
If all the components of $G$ and $H$ are isomorphic to a complete graph, then $\varpi(G\oplus H)\ge 1=\varpi(G)\varpi(H)$. If $G$ or $H$, say $G$, is an empty graph, then for any twins-free clique $S$ in $H$, and any $x\in V(G)$, the set $\{x\}\times S$, is also a twins-free clique in $G\oplus H$, since the adjacencies in each copy of $H$  remains equal and, as a consequence, the inequality $\varpi(G\oplus H)\ge \varpi(G) \varpi(H)$ holds.

From now on, we assume $G$ and $H$ are non-empty graphs and we consider the case that at least one component of $G$ or $H$ is not isomorphic to a complete graph (notice that if at least one component of a graph is not isomorphic to a complete graph, then its twins-free clique number is greater than one). Let $W$ be a $\varpi(G)$-set and let $Z$ be a $\varpi(H)$-set. From the definition of Cartesian sum graphs, we have that the subgraph induced by $W\times Z$ is a clique in $G\oplus H$. We consider the following cases.\\

\noindent Case 1: either $G$ or $H$, say $G$, has every component isomorphic to a complete graph. Hence, $W$ is a singleton set,  $W = \{u\}$, and the set $Z$ is included in  a component of $H$ which is not isomorphic to a complete graph (if not, then $\varpi(H)=1$, which is not possible). So, there exist $v,y\in Z$, $z\notin Z$, such that $z\in N_{H}(v)-N_{H}[y]$. By the definition of Cartesian sum graphs, we obtain that $(u,z)\sim (u,v)$ and $(u,z)\not\sim (u,y)$. Thus, $W\times Z$ is a twins-free clique.\\

\noindent Case 2: neither $G$ nor $H$ has every component isomorphic to a complete graph. Thus, as above, there exist $u,x\in W$ and $w\notin W$ such that $w\in N_{G}(u)-N_{G}[x]$. Also, there exist $v,y\in Z$ and $z\notin Z$ such that $z\in N_{H}(v)-N_{H}[y]$. Again, from the definition of Cartesian sum graphs, we have that
\begin{center}
$(w,z),(u,z),(x,z),(w,v),(w,y)\in N_{G\oplus H}[(u,v)]$,
\end{center}
\begin{center}
$(u,z),(w,v)\in N_{G\oplus H}[(x,y)]$ and $(w,z),(x,z),(w,y)\notin N_{G\oplus H}[(x,y)]$,
\end{center}
\begin{center}
$(w,z),(x,z),(w,v),(w,y)\in N_{G\oplus H}[(u,y)]$ and $(u,z)\notin N_{G\oplus H}[(u,y)]$,
\end{center}
\begin{center}
$(w,z),(u,z),(x,z),(w,y)\in N_{G\oplus H}[(x,v)]$ and $(w,v)\notin N_{G\oplus H}[(x,v)]$.
\end{center}
Therefore, $W\times Z$ is a twins-free clique in $G\oplus H$, which completes the proof.
\end{proof}

Notice that there are cases of Cartesian sum graphs not satisfying the equality in the result above. One example is obtained as a consequence of Corollary \ref{cor varpi cases} considering the graph $K_{1,n}\oplus K_{n'}$.

The clique number of any Cartesian sum  graph satisfies the following relationship.

\begin{lemma}\label{lem omega}
For any graphs $G$ and $H$,
$$\omega(G\oplus H) = \omega(G) \omega(H).$$
\end{lemma}

\begin{proof}
Let $W$ be an $\omega(G)$-set and let $Y$ be an $\omega(H)$-set. From the definition of Cartesian sum graphs, we have that the subgraph induced by $W\times Y$ is a clique in $G\oplus H$. So, $\omega(G\oplus H)\ge \omega(G) \omega(H)$. Let $Z$ be an $\omega(G\oplus H)$-set and let $(u,v)\in Z$. Thus, by using definition of Cartesian sum graphs, $Z$ must be of the form $R\times S$, where $R$ is maximum clique in $G$ containing $u$ and $S$ is maximum clique in $H$ containing $v$. Hence, $\omega(G\oplus H) = |R|\cdot |S|\le \omega(G) \omega(H)$ and the equality holds.
\end{proof}

From now on we present our results on the strong metric dimension of Cartesian sum graphs. Notice that the graphs $G\oplus H$ having diameter two are described in  Proposition \ref{lem Cart sum diam}.

\begin{proposition}\label{pro bounds cart sum}
Let $G$ and $H$ be two graphs of order $n$ and $n'$, respectively, such that $G\oplus H$ is connected. Then,
$$\dim_s(G\oplus H)\le n n' - \varpi(G) \varpi(H).$$
Moreover, if   $D(G)\le 2$
or neither $G$ nor  $H$ has isolated vertices, then
$$n n' - \omega(G) \omega(H)\le \dim_s(G\oplus H)\le n n' - \varpi(G) \varpi(H).$$
\end{proposition}

\begin{proof}
From Theorem \ref{th D2}, Lemma \ref{lem omega} and the fact that $\omega(H)\ge \varpi(H)$, we have the lower bound. On the other hand, the upper bounds hold because of Theorem \ref{th D2} and Lemma \ref{lem varpi}.
\end{proof}

\begin{corollary}\label{CorollaryBothCliquesEqual}
Let $G$ and $H$ be two graphs of order $n$ and $n'$, respectively, such that $D(G)\le 2$
or neither $G$ nor  $H$ has isolated vertices. If $\omega(G) =\varpi(G)$ and $\omega(H)=\varpi(H)$, then
$$ \dim_s(G\oplus H)=n n' - \omega(G) \omega(H).$$
\end{corollary}

We recall that the \emph{fan graph} $F_{1,n}$ is defined as the graph join $K_1 + P_n$, the \emph{wheel graph} of order $n+1$ is defined as $W_{1,n} = K_1 + C_n$ and the \textit{grid graph} $P_n\Box P_{n'}$ is obtained as the Cartesian product of the paths $P_n$ and $P_{n'}$. There are some families of graph, as the above ones,  which have no true twin vertices. In this sense, its twins-free clique number is equal to its clique number \textit{i.e.},

\begin{itemize}
\item $\varpi(T_n) = \omega(T_n) = 2$, where $T_n$ is a tree of order $n\ge 2$.
\item $\varpi(C_n) = \omega(C_n) = 2$, where $n\ge 3$.
\item $\varpi(F_{1,n}) = \omega(F_{1,n}) = 3$, where $n\ge 3$.
\item $\varpi(W_{1,n}) = \omega(W_{1,n}) = 3$, where $n\ge 4$.
\item  $\varpi(P_n\Box P_{n'}) = \omega(P_n\Box P_{n'}) = 2$, where $n,n'\ge 2$.
\end{itemize}

By using the examples above, Corollary \ref{CorollaryBothCliquesEqual} leads to the following.

\begin{remark} \label{examples} The following assertions hold.
\begin{enumerate}[{\rm(i)}]
\item If $G$ and $H$ are trees, cycles or grid graphs of order $n$ and $n'$, respectively, then $$\dim_s(G\oplus H) = n n' - 4.$$
\item If $G$ and $H$ are fans or wheels of order $n+1$ and $n'+1$, respectively, then $$\dim_s(G\oplus H) = n n' + n + n' - 8.$$
\item If $G$ is a tree, a cycle or a grid graph of order $n$ and $H$ is a fan or a wheel of order $n'+1$, then $$\dim_s(G\oplus H) = n n' + n - 6.$$
\end{enumerate}
\end{remark}

Lemma \ref{lem varpi} gives a general lower bound for $\varpi(G\oplus H)$ in terms of $\varpi(G)$ and $\varpi(H)$. Next we give another lower bound, which  in some cases behaves better than the one from Lemma \ref{lem varpi}. A \emph{simplicial vertex} in a graph $G$ is a vertex of degree order of $G$ minus one.

\begin{lemma}\label{lem varpi cases}
Let $G$ and $H$ be two non-trivial graphs.
Then $$\varpi(G\oplus H)\ge \max\{(\varpi(G)-1) \omega(H),(\varpi(H)-1) \omega(G)\} + 1.$$
Moreover, if there exists a $\varpi(G)$-set without vertices of degree order minus one, then $$\varpi(G\oplus H)\ge \varpi(G) \omega(H).$$
\end{lemma}

\begin{proof}
Assume $G$ has order $n$ and let $W$ be a $\varpi(G)$-set without vertices of degree $n-1$ and let $Z$ be a $\omega(H)$-set. From the definition of Cartesian sum graphs, we have that the subgraph induced by $W\times Z$ is a clique in $G\oplus H$. Let $(u,v)$ and $(x,y)$ be two different vertices belonging to $W\times Z$. In order to show that $W\times Z$ is a twins-free clique, we consider the following cases.
\\
\\
\noindent Case 1: $v=y$. Since $u,x\in W$, then without loss of  generality, there exists vertex $w\in N_{G}(u)-N_{G}[x]$. Hence,  $(u,v)\sim (w,v)\not\sim (x,y)$.
\\
\\
\noindent Case 2: $v\ne y$. Since $u$ has degree less than or equal to $n-2$, there exists vertex $z\in V(G)$ such that $u\not\sim z$. Thus,  $(u,v)\not\sim (z,v)\sim (x,y)$.

Thus, $W\times Z$ is a twins-free clique and so $\varpi(G\oplus H)\ge |W\times Z|= \varpi(G) \omega(H)$.

On the other hand, let $Y$ be $\varpi(G)$-set having a vertex $a$ of degree $n-1$. Notice that $Y$ cannot contain other vertex of degree $n-1$. Now, let $b$ be a vertex belonging to $Z$. Observe that $S=((Y-\{a\})\times Z) \cup \{(a,b)\}$ is also a clique in $G\oplus H$ since $Y\times Z$ is a clique. We claim that $S$ is a twins-free clique. To see this,  we differentiate the following cases for two different vertices $(c,d),(e,f)\in S$.
\\
\\
\noindent
Case 1': $d=f$. Proceeding like in Case 1, we have that $(c,d)$ and $(e,f)$ are not true twins.
\\
\\
\noindent Case 2': $d\ne f$.  If  $c\ne a$, then  $c$ has degree less than or equal to $n-2$ and there exists a vertex $g\in V(G)$ such that $c\not\sim g$. Thus,  $(c,d)\not\sim (g,d)\sim(e,f)$. Now, suppose that $c=a$. In this case   $d=b$ and $e\ne a$. Since there exists $a'\in V(H)$ such that $a'\in N_H(a)-N_H[e]$, we have $(c,d)=(a,b)\sim (a',f)\not \sim (e,f)$.

Therefore, $S$ is a twins-free clique, which leads to
$$\varpi(G\oplus H)\ge |S|= (\varpi(G)-1) \omega(H) + 1.$$
By the symmetry of the Cartesian sum graphs we complete the proof.
\end{proof}

The following result is a direct consequence of the lemma above and the well known fact that the Cartesian sum of graphs is a commutative operation.

\begin{corollary}\label{cor varpi cases}
Let $G$ and $H$ be two non-trivial graphs of order $n$ and $n'$, respectively. Then the following assertions hold.
\begin{enumerate}[{\rm (i)}]
\item $\varpi(G\oplus H)\ge \max \{(\varpi(G)-1) \omega(H), \omega(G) (\varpi(H)-1)\}+1.$
\item If there exists a $\varpi(G)$-set without a vertex of degree $n-1$ and there exists a $\varpi(H)$-set without a vertex of degree $n'-1$, then
$$\varpi(G\oplus H)\ge \max \{\varpi(G) \omega(H),\omega(G) \varpi(H)\}.$$

\item If there exists a $\varpi(G)$-set without a vertex of degree $n-1$, then
$$\varpi(G\oplus H)\ge \max \{\varpi(G) \omega(H),\omega(G) (\varpi(H)-1) + 1\}.$$
\end{enumerate}
\end{corollary}

\begin{proposition}\label{pro other upper bounds cart sum}
Let $G$ and $H$ be two non-trivial graphs of order $n$ and $n'$, respectively such that $G\oplus H$ is connected.  Then the following assertions hold.
\begin{enumerate}[{\rm (i)}]
\item
$\dim_s(G\oplus H)\le nn' - \max \{(\varpi(G)-1) \omega(H), \omega(G) (\varpi(H)-1)\}-1.$
\item If there exists a $\varpi(G)$-set without a vertex of degree $n-1$ and there exists a $\varpi(H)$-set without a vertex of degree $n'-1$, then
$$\dim_s(G\oplus H)\le nn' - \max \{\varpi(G) \omega(H),\omega(G) \varpi(H)\}.$$

\item If there exists a $\varpi(G)$-set without a vertex of degree $n-1$, then
$$\dim_s(G\oplus H)\le nn' - \max \{\varpi(G) \omega(H),\omega(G) (\varpi(H)-1) + 1\}.$$
\end{enumerate}
\end{proposition}

\begin{proof}
The bounds holds because of Theorem \ref{th D2} and Corollary \ref{cor varpi cases}.
\end{proof}

\begin{corollary}\label{cor G and complete}
Let $G$ be a non-trivial graph of order $n$ and maximum degree $\Delta$. If $G$ has no true twins and $\Delta\le n-2$, then
$$\dim_s(G\oplus K_{n'}) = nn' - n' \omega(G).$$
\end{corollary}

\begin{proof}
First of all, note that $G\oplus K_{n'}$ is connected, as stated in Corollary \ref{rem diam2}. On the other hand, since $G$ has no true twins, it follows $\omega(G)=\varpi(G)$. Now, from Proposition \ref{pro bounds cart sum} we have that $\dim_s(G\oplus K_{n'})\ge nn' - n'\omega(G)$. Moreover, by using Proposition \ref{pro other upper bounds cart sum} (iii) we obtain $\dim_s(G\oplus H)\le nn' - \max \{\varpi(G) \omega(H),\omega(G) (\varpi(H)-1) + 1\} = nn' - n'\omega(G)$. Therefore, the equality holds.
\end{proof}

\begin{corollary}\label{cor star and complete} For any integers, $n,n'\ge 2$,
$$(n+1)n'-2n' \le \dim_s(K_{1,n}\oplus K_{n'}) \le (n+1)n' - n' - 1.$$
\end{corollary}

\begin{proof}
The lower bound is a direct consequence of Proposition \ref{pro bounds cart sum} while the upper bound is a direct consequence of Proposition \ref{pro other upper bounds cart sum} (i).
\end{proof}

\section*{Conclusion and open problems}

We have studied the strong metric dimension of Cartesian sum graphs $G\oplus H$ for all the possibilities of connectivity of $G\oplus H$ with the following exception: when $G$ has an isolated vertex, $H$ has no isolated vertices and $D(H)>2$. That is, it remains to study the case where  $G\oplus H$ has diameter three. We leave this case as an open problem. Also, according to the open problem presented in \cite{Kuziak2014} about characterizing some  kind of strong resolving graphs, in this article we describe the strong resolving graphs of Cartesian sum graphs of diameter two.


\begin{thebibliography}{10}
\expandafter\ifx\csname url\endcsname\relax
  \def\url#1{\texttt{#1}}\fi
\expandafter\ifx\csname urlprefix\endcsname\relax\def\urlprefix{URL }\fi

\bibitem{Cizek1994}
N.~\v{C}i\v{z}ek, S.~Klav\v{z}ar, On the chromatic number of the lexicographic
  product and the {C}artesian sum of graphs, Discrete Mathematics 134~(1-3)
  (1994) 17--24.

\bibitem{Hammack2011}
R.~Hammack, W.~Imrich, S.~Klav{\v{z}}ar, Handbook of product graphs, Discrete
  Mathematics and its Applications, 2nd ed., CRC Press, 2011.

\bibitem{Harary1976}
F.~Harary, R.~A. Melter, On the metric dimension of a graph, Ars Combinatoria 2
  (1976) 191--195.

\bibitem{Johnson1993}
M.~Johnson, Structure-activity maps for visualizing the graph variables arising
  in drug design, Journal of Biopharmaceutical Statistics 3~(2) (1993)
  203--236, pMID: 8220404.

\bibitem{Johnson1998}
M.~Johnson, Browsable structure-activity datasets, in: R.~Carb\'{o}-Dorca,
  P.~Mezey (eds.), Advances in Molecular Similarity, chap.~8, JAI Press Inc,
  Stamford, Connecticut, 1998, pp. 153--170.

\bibitem{Khuller1996}
S.~Khuller, B.~Raghavachari, A.~Rosenfeld, Landmarks in graphs, Discrete
  Applied Mathematics 70~(3) (1996) 217--229.

\bibitem{Kuziak2013}
D.~Kuziak, I.~G. Yero, J.~A. Rodr\'{\i}guez-Vel\'{a}zquez, On the strong metric
  dimension of corona product graphs and join graphs, Discrete Applied
  Mathematics 161~(7--8) (2013) 1022--1027.

\bibitem{Kuziak2013b}
D.~Kuziak, I.~G. Yero, J.~A. Rodr\'iguez-Vel\'azquez, Strong metric dimension
  of rooted product graphs, International Journal of Computer Mathematics. (2015) To
  appear. http://dx.doi.org/10.1080/00207160.2015.1061656

\bibitem{Kuziak2014}
D.~Kuziak, I.~G. Yero, J.~A. Rodr\'iguez-Vel\'azquez, Closed formulae for the
  strong metric dimension of lexicographic product graphs, arXiv:1402.2663v1
  [math.CO].

\bibitem{Kuziak2013c}
D.~Kuziak, I.~G. Yero, J.~A. Rodr{\'{\i}}guez-Vel{\'a}zquez, On the strong
  metric dimension of the strong products of graphs, Open Mathematics 13 (2015)
  64--74.
  
\bibitem{Kuziak-Erratum}
D.~Kuziak, I.~G. Yero, J.~A. Rodr{\'{\i}}guez-Vel{\'a}zquez, Erratum to ``{O}n
  the strong metric dimension of the strong products of graphs'', Open Mathematics 13
  (2015) 209--210.

\bibitem{Melter1984}
R.~A. Melter, I.~Tomescu, Metric bases in digital geometry, Computer Vision,
  Graphics, and Image Processing 25~(1) (1984) 113--121.

\bibitem{Oellermann2007}
O.~R. Oellermann, J.~Peters-Fransen, The strong metric dimension of graphs and
  digraphs, Discrete Applied Mathematics 155~(3) (2007) 356--364.

\bibitem{Ore1962}
O.~Ore, Theory of Graphs, Colloquium Publications, Volume 38, American
  Mathematical Society, 1962.
  
\bibitem{Rodriguez-Velazquez2013a} J. A. Rodr\'iguez-Vel\'azquez, I. G. Yero, D. Kuziak, O. R. Oellermann,
On the strong metric dimension of Cartesian and direct products of graphs, Discrete Mathematics 335 (2014) 8--19.

\bibitem{Scheinerman1997}
E.~R. Scheinerman, D.~H. Ullman, Fractional Graph Theory, Series in Discrete
  Mathematics and Optimization, Wiley-Interscience, 1997.

\bibitem{Sebo2004}
A.~Seb\"{o}, E.~Tannier, On metric generators of graphs, Mathematics of
  Operations Research 29~(2) (2004) 383--393.

\bibitem{Shannon1956}
C.~E. Shannon, The zero error capacity of a noisy channel, IRE Transactions on
  Information Theory 2~(3) (1956) 8--19.

\bibitem{Slater1975}
P.~J. Slater, Leaves of trees, Congressus Numerantium 14 (1975) 549--559.

\end{thebibliography}
\end{document}